\documentclass[12pt]{article}
\usepackage{amsmath,amsthm,amsfonts}
\usepackage{graphicx}
\newcommand{\Z}{\mathbb{Z}}

\newcommand{\C}{\mathbb C}

\newcommand{\g}{\lambda}
\newcommand{\ze}{\bold 0}

\newtheorem{theorem}{Theorem}
\newtheorem{corollary}{Corollary}
\newtheorem{lemma}{Lemma}

\begin{document}
\title
{Some resolvent set  properties of band operators with matrix
elements.}
\author{ Andrey Osipov \thanks{Supported by
RFBR: project 14-01-00349 }}
\date{}
\maketitle

\begin{abstract}
For operators generated by a certain class of infinite band
matrices with matrix elements we establish a characterization of
the resolvent set in terms of polynomial solutions of the
underlying higher order finite difference equations. This enables
us to describe some asymptotic behaviour of the corresponding
systems of vector orthogonal polynomials on the resolvent set.
\end{abstract}

\section{Introduction}

In recent years some results concerning the structure of the
resolvent set of nonsymmetric difference operators generated by
infinite band matrices were obtained ~\cite{akj}- ~\cite{ b2}.
Originally the second order difference operators generated by
three-diagonal matrices were studied ~\cite{akj} - ~\cite{b},
later on some results on the resolvent set properties of higher
order band operators were obtained ~\cite{cal} - ~\cite{ b2}.
 In studying the spectral properties of such band operators the
important role is played by the Weyl matrix ~\cite{ b2} (or the
Weyl function in the second order case ~\cite{akj}- ~\cite{ cb})
of the corresponding operator, the systems of polynomials
orthogonal with respect to the Weyl functions as well as the
convergence of the Hermite-Pade approximations of the Weyl matrix.
Here we obtain some results similar to ~\cite{ b2} in the case of
higher order band operators with matrix elements. Also note that
for the second order band operators with matrix elements a similar
results on the structure of their resolvent sets were obtained in
~\cite{os1}.


 Consider an infinite nonsymmetric band
matrix $A=(A_{k,l})_{k,l=0}^\infty$ whose entries are matrices of
order $N$ with complex elements: $A_{k,l} \in \mathbb{C}^{N\times
N}$, satisfying for all $k$
 and for all $\ell < k-s$ or $\ell
> k+r$
\begin{eqnarray}
\label{a} A_{k,\ell}=O, \quad A_{k,k+r}, \, A_{k+s,k} \quad\text{are
invertible},
\end{eqnarray}
where $O$ is a zero matrix of order $N$.
 That is, $A$ takes the form
\begin{eqnarray*}
A=
\left(
\begin{matrix}
A_{0,0}&\dots&A_{0,r}&O&O&\dots&\dots&\dots\\
    A_{1,0}&A_{1,1}&\dots&A_{1,r+1}&O&\dots&\dots&\dots\\
   \vdots&\vdots&\vdots&\vdots&\ddots&\ddots&\\
A_{s,0}&A_{s,1}&A_{s,2}&\dots&\dots&A_{s,s+r}&O&\dots\\
 O&A_{s+1,1}&A_{s+1,2}&A_{s+1,3}&\dots&\dots&A_{s+1,s+r+1}&\dots\\
O&O&\ddots&\vdots&\vdots&\vdots&\vdots&\ddots
\end{matrix}
\right)
\end{eqnarray*}
with $s+r+1$ nontrivial diagonals, where $r $ and $s$ are some
fixed natural numbers.
To the matrix $A$ we assign finite-difference equations in the
matrices $Y_k$, $Y_k^+ \in \mathbb{C}^{N\times N}$
\begin{eqnarray}
&&\label{d1} A_{k,k-s}Y_{k-s}+A_{k,k-s+1}Y_{k-s+1}+
\dots+A_{k,k+r}Y_{k+r}=\g Y_k,
\\&&\label{d2} Y_{k-r}^{+} A_{k-r,k}+Y_{k-r+1}^{+}A_{k-r+1,k}+\dots +
Y_{k+s}^{+}A_{k+s,k}=\g Y_{k}^+,
\end{eqnarray}
where $k\ge 0$, $\g\in\C$ is some parameter, and we define
$A_{k,\ell}$ with negative indices as
\begin{eqnarray*}
&A_{k,k-s}=-E,\quad A_{k-s+j,k}= O,\quad 0\le k<s, \quad 1\le j
<s-k;\\
&A_{k-r,k}=-E,\quad A_{k-r+j,k}=O,\quad 0\le k < r,\; 1\le j <r-k;
\end{eqnarray*}

where $E$ is a unit matrix.

We consider some particular fundamental systems of solutions
$\{P(\g),Q(\g)\}$ of (\ref{d1}), and $\{ P^{+}(\g),Q^{+}(\g)\}$ of
(\ref{d2}), respectively, with elements being polynomials with
respect to $\g$ with matrix coefficients: Denote by $$
\begin{gathered}
 Q(\g)=(Q_k(\g))_{k=-s}^{\infty}=
 (Q_k^{1}(\g),\dots,Q_k^{r}(\g))_{k=-s}^{\infty};\\
 P(\g)=(P_k(\g))_{k=-s}^{\infty}=
 (P_k^{1}(\g),\dots,P_k^{s}(\g))_{k=-s}^{\infty};\\
\end{gathered}
$$ solutions of (\ref{d1}) satisfying the initial conditions
\begin{equation} \label{init_L}
    Q_{0:r-1} = I_r , \, Q_{-s:-1} = \ze_{s\times r} , \,
    P_{-s:-1} = I_s , \,
    P_{0:r-1} = \ze_{r\times s} .
\end{equation}
Furthermore, denote by $$
\begin{gathered}
Q^+(\g)= \left(
\begin{array}{c}
Q_k^{1,+}(\g)\\ \vdots\\ Q_k^{s,+}(\g)
\end{array}
\right)_{k=-r}^\infty , \quad P^+(\g)= \left(
\begin{array}{c}
P_k^{1,+}(\g)\\ \vdots\\ P_k^{r,+}(\g)
\end{array}
\right)_{k=-r}^\infty
\end{gathered}
$$ solutions of the dual recurrence relation (\ref{d2}) satisfying
the initial conditions \begin{equation} \label{init_R}
\begin{gathered}
    Q_{0:s-1}^+ = I_s , \, Q_{-r:-1}^+ = \ze_{s\times r} , \,
    P_{-r:-1}^+ = I_r , \,
    P_{0:s-1}^+ = \ze_{r\times s} .
\end{gathered}
\end{equation}
Here and in what follows we use the following notations: $Q_{j:k}$
 (and $Q_{j:k}^+$, respectively)
for the stacked matrix with rows $Q_\ell$, $\ell=j,j+1,...,k$ (with
columns $Q_\ell^+$, $\ell=j,j+1,...,k$), etc; $I_i, \ze_{i\times j}$
for the identity and zero matrices of sizes $i$ and $i\times j$
(with the elements from $\mathbb{C}^{N\times N}$, $\, I_1=E$). We
shall also use $A_{j:k,m:n}$ for the submatrix of $A$ composed of
its rows labeled $j$ to $k$, and its columns labeled $m$ to $n$.
Finally, we use block matrix notations like $[M \; N]$.

\begin{lemma}
 \label{lem1}
 The expression
 $$
     F(k) \equiv \sum_{i=0}^{s-1}\sum_{j=i+1}^s Y_{k+i}^{+}
    A_{k+i,k+i-j}Y_{k+i-j}-\sum_{i=0}^{r-1}\sum_{j=i+1}^r
     Y_{k+i-j}^{+}A_{k+i-j,k+i}Y_{k+i}, \quad k\ge 0
  $$
  does not depend on $k$
\end{lemma}
\begin{proof}
We have
\begin{eqnarray*}
F(k)&=&Y_k^+ A_{k,k-1}Y_{k-1}+Y_k^+ A_{k,k-2}Y_{k-2}+\dots +
Y_k^{+}A_{k,k-s}Y_{k-s}+\\
&+&Y_{k+1}^+ A_{k+1,k-1}Y_{k-1}+Y_{k+1}^+ A_{k+1,k-2}Y_{k-2}+\dots +
Y_{k+1}^{+}A_{k+1,k+1-s}Y_{k+1-s}+\\
+ \dots  &+& Y_{k+s-1}^+ A_{k+s-1,k-1}Y_{k-1}-Y_{k-1}^+
A_{k-1,k}Y_{k}-Y_{k-2}^+ A_{k-2,k}Y_{k}-\dots -
Y_{k-r}^{+}A_{k-r,k}Y_{k}-\\
&-&Y_{k-1}^+ A_{k-1,k+1}Y_{k+1}-Y_{k-2}^+ A_{k-2,k+1}Y_{k+1}-\dots -
Y_{k-r+1}^{+}A_{k-r+1,k}Y_{k+1}-\\
&-& \dots  - Y_{k-1}^+ A_{k-1,k+r-1}Y_{k+r-1}.
\end{eqnarray*}
  Applying (\ref{d1})- (\ref{d2}) to $F(k)$ for $i=0$, we
  get
 \begin{eqnarray*}
F(k)&=&Y_k^+(\g-A_{k,k})Y_{k}-Y_k^+ A_{k,k+1}Y_{k+1}-\dots
-Y_k^+ A_{k,k+r}Y_{k+r}+\\
&+&Y_{k+1}^+ A_{k+1,k-1}Y_{k-1}+Y_{k+1}^+ A_{k+1,k-2}Y_{k-2}+\dots +
Y_{k+1}^{+}A_{k+1,k+1-s}Y_{k+1-s}+\\
+ \dots  &+& Y_{k+s-1}^+ A_{k+s-1,k-1}Y_{k-1}+Y_k^+(A_{k,k}-\g)Y_k +
Y_{k+1}^+ A_{k+1,k}Y_{k}+\dots +Y_{k+s}^+ A_{k+s,k}Y_{k}-\\
&-&Y_{k-1}^+ A_{k-1,k+1}Y_{k+1}-Y_{k-2}^+ A_{k-2,k+1}Y_{k+1}-\dots -
Y_{k-r+1}^{+}A_{k-r+1,k}Y_{k+1}-\\
&-& \dots  - Y_{k-1}^+ A_{k-1,k+r-1}Y_{k+r-1}.
 \end{eqnarray*}
 By separating the ``positive" and ``negative" parts of $F(k)$, we
 obtain
 \begin{eqnarray*}
F(k)&=&- Y_k^+ A_{k,k+2}Y_{k+2}-\dots
-Y_k^+ A_{k,k+r}Y_{k+r}+\\
&+&Y_{k+1}^+ A_{k+1,k-1}Y_{k-1}+Y_{k+1}^+ A_{k+1,k-2}Y_{k-2}+\dots +
Y_{k+1}^{+}A_{k+1,k+1-s}Y_{k+1-s}+\\
+ \dots  &+& Y_{k+s-1}^+ A_{k+s-1,k-1}Y_{k-1}+ \\
&+& Y_{k+2}^+ A_{k+2,k}Y_{k}+\dots +Y_{k+s}^+ A_{k+s,k}Y_{k}- \\
&-& Y_k^+ A_{k,k+1}Y_{k+1} - Y_{k-1}^+ A_{k-1,k+1}Y_{k+1}- \dots
-Y_{k-r+1}^{+}A_{k-r+1,k}Y_{k+1}- \\
&-& \dots  - Y_{k-1}^+ A_{k-1,k+r-1}Y_{k+r-1}= \dots \\
&=&Y_{k+1}^+ A_{k+1,k}Y_{k}+Y_{k+1}^+ A_{k+1,k-1}Y_{k-1}+\dots +
Y_{k+1}^{+}A_{k+1,k+1-s}Y_{k+1-s}+\\
&+&Y_{k+2}^+ A_{k+2,k}Y_{k}+Y_{k+2}^+ A_{k+2,k-1}Y_{k-1}+\dots +
Y_{k+2}^{+}A_{k+2,k+2-s}Y_{k+2-s}+\\
+ \dots  &+& Y_{k+s}^+ A_{k+s,k}Y_{k}-Y_{k}^+
A_{k,k+1}Y_{k+1}-Y_{k-1}^+ A_{k-1,k+1}Y_{k+1}-\dots -
Y_{k-r+1}^{+}A_{k-r+1,k+1}Y_{k+1}-\\
&-&Y_{k}^+ A_{k,k+2}Y_{k+2}-Y_{k-1}^+ A_{k-1,k+2}Y_{k+2}-\dots -
Y_{k-r+2}^{+}A_{k-r+2,k+1}Y_{k+2}-\\
&-& \dots  - Y_{k}^+ A_{k,k+r}Y_{k+r}=F(k+1).
\end{eqnarray*}
\end{proof}
\begin{lemma}
 \label{lem2}
For all $k \geq 0$
  \begin{eqnarray}
  \label{mi}
   && I_{r+s} = \left[\begin{array}{c}
   P^+_{k-r:k+s-1}  \\ Q^+_{k-r:k+s-1}
   \end{array}\right]\times
   \\&&
  \times
   \left[\begin{array}{cc}
      \bf{0}_{r \times s} &
      - A_{k-r:k-1,k:k+r-1} \\
      A_{k:k+s-1,k-s:k-1} & \bf{0}_{s \times r}
   \end{array}\right]
   \times
   \left[\begin{array}{cc}
   Q_{k-s:k+r-1} , - P_{k-s:k+r-1}
   \end{array}\right]
\end{eqnarray}
\end{lemma}
\begin{proof}
  For $k=0$, the claim follows  from the choice
  (\ref{init_L}) and (\ref{init_R}) of the initial
  conditions and the definition of $A_{k,\ell}$ with negative indices.
  Now consider the particular solutions
  $Y_n=Q_{n}^{m}$ and
  $Y_n^+=P_{n}^{\ell,+}$ for some indices $\ell,m=1,\dots,r$.
  We find, by using  (\ref{a}) that
  \begin{eqnarray*}
G^{\ell, m}(k) \equiv \sum_{i=0}^{s-1}\sum_{j=1}^s
P_{k+i}^{\ell,+}A_{k+i,k+i-j}Q_{k+i-j}^{m}-
\sum_{i=0}^{r-1}\sum_{j=1}^r
P_{k+i-j}^{\ell,+}A_{k+i-j,k+i}Q_{k+i}^{m}=\\
= \sum_{i=0}^{s-1}\sum_{j=i+1}^s
P_{k+i}^{\ell,+}A_{k+i,k+i-j}Q_{k+i-j}^{m}-
\sum_{i=0}^{r-1}\sum_{j=i+1}^r
P_{k+i-j}^{\ell,+}A_{k+i-j,k+i}Q_{k+i}^{m}
\end{eqnarray*}
According to Lemma~1, for $k>0$, $G^{\ell, m}(k)=G^{\ell,
m}(0)=\delta _{\ell, m}$
 In other words, we have shown
that the corresponding entries in the first $r$ rows and columns
of the claimed matrix identity coincide. The identities for the
other three blocks : $r\times s, \; s\times r$ and $s \times s$
are obtained in a similar way by choosing $Y_n\in
\{P_{n}^{m},Q_{n}^{m}\}$ and $Y_n^+\in
\{P_{n}^{\ell,+},Q_{n}^{\ell,+}\}$.
\end{proof}

\section{Band operators and their resolvent set properties}

The above matrix $A$ generates a linear operator in the space
$l^2_N$ of sequences $u=(u_0,u_1,\dots)\; $ where the vector column
$u_j \in \mathbb{C}^N, $ with inner product $ (u,v)=
\sum_{j=0}^{\infty} v_j^*u_j \,$. For this operator we shall use the
same notation.

Let $I$ be the identity operator in $\,l^2_N \,$.Then it admits the
matrix representation
\begin{eqnarray*}
I= \left(
\begin{matrix}
E&O&O&\dots\\
    O&E&O&\dots\\
 O&O&E&\dots\\
   \vdots&\vdots&\vdots&\ddots\\
\end{matrix}
\right)
\end{eqnarray*}

Let $\;\mathfrak M= (\mathfrak M_{i,j})_{i=1,\dots,r}^{j=1,\dots,s},
\; \mathfrak M_{i,j} \in \mathbb C^{N \times N} \;$ be an arbitrary
matrix of the size $\,r \times s \,$ with matrix elements. Then for
$\,k,n \in \mathbb{Z_+}\,$ we define

\begin{eqnarray}
 \label{r1}
 R_{k,n}=\begin{cases}
  Q_k(\g)R_n^+(\g), \quad
  0\le k < n+r, \\
  R_k(\g)Q_n^+(\g), \quad 0 \le n < k+s
 \end{cases}
 \end{eqnarray}

 where $R_k(\g)=Q_k(\g)\mathfrak M -P_k(\g)$ and
 $R_n^+(\g)= \mathfrak M Q_n^+(\g)-P_n^+(\g)$.

 As we see, we have two different definitions of $R_{k,n}$ for $\,
 n-s< k <n+r \,$. In fact, they give the same value.

 \begin{lemma}
 \label{lem3}
For $ n-s <k <n+r $   $ \, Q_k(\g) R_n^+(\g)=R_k(\g)Q_n^+(\g).$
 \end{lemma}
\begin{proof}
As follows from the definition of $R_k(\g), \,$ $R_n^+ (\g), \,$ it
suffices to show that
\begin{equation}
\label{rel}
 Q_k(\g) P_n^+(\g)=P_k(\g)Q_n^+(\g),\quad n-s <k <n+r .
\end{equation}
Consider the case $\,s=r=1\,$. By Lemma~2 we have
\begin{eqnarray}
\label{qq}
\begin{cases}
P_k^+(\g)A_{k,k-1}Q_{k-1}(\g)-P_{k-1}^+(\g)A_{k-1,k}Q_{k}(\g)=E \\
P_k^+(\g)A_{k,k-1}P_{k-1}(\g)-P_{k-1}^+(\g)A_{k-1,k}P_{k}(\g)=O \\
Q_k^+(\g)A_{k,k-1}Q_{k-1}(\g)-Q_{k-1}^+(\g)A_{k-1,k}Q_{k}(\g)=O \\
-Q_k^+(\g)A_{k,k-1}P_{k-1}(\g)+Q_{k-1}^+(\g)A_{k-1,k}P_{k}(\g)=E.
\end{cases}
\end{eqnarray}
Now for $ k \ge 0 $ consider the system
\begin{eqnarray*}
\begin{cases}
Q_{k-1}(\g)\Delta_k^1 + P_{k-1}(\g)\Delta_k^2=O\\
A_{k-1,k}Q_k(\g)\Delta_k^1+A_{k-1,k}P_k(\g)\Delta_k^2=E,
\end{cases}
\end{eqnarray*}
where $\Delta_k^1, \, \Delta_k^2 \,$ are unknown. Multiplying the
first equation of the system on the left by $P_k^+(\g)A_{k,k-1} $
and the second equation by $-P_{k-1}^+(\g)\,$ and summing the
resulting equations, we obtain
\begin{eqnarray*}
(P_k^+(\g)A_{k,k-1}Q_{k-1}(\g)-P_{k-1}^+(\g)A_{k-1,k}Q_{k}(\g))\Delta_k^{1}+\\
+(P_k^+(\g)A_{k,k-1}P_{k-1}(\g)-P_{k-1}^+(\g)A_{k-1,k}P_{k}(\g))\Delta_k^{2}=
-P_{k-1}^+(\g).
\end{eqnarray*}
Applying (\ref{qq}) we find $\,\Delta_k^{1}= - P_{k-1}^+(\g).$

 Similarly, multiplying the first equation of the system on the
left by $\,-Q_k^+(\g)A_{k,k-1}\,$ and the second equation by
$\,Q_{k-1}(\g)$ we find $\,\Delta_k^{2}= Q_{k-1}^+(\g).$

Thus $\, Q_n(\g) P_n^+(\g)=P_n(\g)Q_n^+(\g)\,n \ge 0\,$ and
therefore $\, Q_n(\g) R_n^+(\g)=R_n(\g)Q_n^+(\g).$

For an arbitrary $s\,$ and $r\,$ the proof is similar as above (and
based on Lemma~2).

\end{proof}
Now consider the infinite matrix $R=(R_{k,n})_{k,n=0}^\infty$
\begin{lemma}
\label{lem4}
 The following matrix identities (formal products
between infinite matrices)
\begin{equation} \label{mid}
   (\g I-A)R=I , \quad R(\g I-A)=I
\end{equation}
are hold.
\end{lemma}
\begin{proof}
By symmetry (replace $A$ by its transposed), it is sufficient to
show the first identity of (\ref{mid}). Since $(Q_n(\g))_{n \geq
0}$ is a solution of (\ref{d1}), we have
\begin{equation} \label{first_p}
(\g I-A)
      \left[\begin{matrix} R_{0:k-1,n} \\ {\bf 0}_{\infty \times 1} \end{matrix}\right]
    = \left[\begin{array}{c} {\bf 0}_{(k-r) \times 1} \\ A_{k-r:k-1,k:k+r-1}
R_{k:k+r-1,n}
    \\ - A_{k:k+s-1,k-s:k-1} R_{k-s:k-1,n} \\ {\bf 0}_{\infty \times 1}
   \end{array}\right] , \quad r \leq k \leq n .
\end{equation}
Similarly, $(R_n(\g))_{n \geq 0}$ is a solution of (\ref{d1}), and
hence we have the (formal) identity
\begin{equation} \label{second_p}
(\g I-A)
      \left[\begin{matrix} {\bf 0}_{n \times 1} \\ R_{n:\infty,n} \end{matrix}\right]
    = \left[\begin{array}{c} {\bf 0}_{(n-r) \times 1} \\ - A_{n-r:n-1,n:n+r-1}
R_{n:n+r-1,n}
    \\ A_{n:n+s-1,n-s:n-1} R_{n-s:n-1}(\g) Q^+_{n}(\g) \\ {\bf 0}_{\infty \times 1}
   \end{array}\right] .
\end{equation}
Combining identity (\ref{first_p}) for $k=n$ with (\ref{second_p})
leads to the (formal) identity
\begin{equation*}
  (\g I-A) R_{0:\infty,n}
  = \left[\begin{array}{c} {\bf 0}_{n \times 1}
    \\ A_{n:n+s-1,n-s:n-1} [R_{n-s:n-1}(\g) Q^+_{n}(\g) -
    R_{n-s:n-1,n}]
    \\ {\bf 0}_{\infty \times 1} \end{array}\right]
\end{equation*}
Using the definition of $R_{k,n}$ and (\ref{rel}) we have
\begin{equation*}
A_{n:n+s-1,n-s:n-1} [R_{n-s:n-1}(\g) Q^+_{n}(\g) -
    R_{n-s:n-1,n}]=\left[\begin{array}{c}
     A_{n,n-s}(Q_{n-s}(\g)P_n^{+}(\g)-P_{n-s}(\g)Q_n^{+}(\g))
  \\ {\bf 0}_{(s-1) \times 1}
   \end{array}\right]
\end{equation*}
It remains to show that
$$
 A_{n,n-s}(-P_{n-s}(\g)Q_n^{+}(\g)+Q_{n-s}(\g)P_n^{+}(\g))=E.
$$
In doing so, we use the first row of the matrix identity
(\ref{mi}) and (\ref{rel}).

For example, for $s=1$ and $ r=2$ the first row of (\ref{mi})
gives
\begin{eqnarray}
\label{case12} \nonumber
P_n^{1,+}A_{n,n-1}Q_{n-1}^1-(P_{n-2}^{1,+}A_{n-2,n}-P_{n-1}^{1,+}A_{n-1,n})Q_n^1
-P_{n-1}^{1,+}A_{n-1,n+1}Q_{n+1}^1=E\\
P_n^{1,+}A_{n,n-1}Q_{n-1}^2-(P_{n-2}^{1,+}A_{n-2,n}-P_{n-1}^{1,+}A_{n-1,n})Q_n^2
-P_{n-1}^{1,+}A_{n-1,n+1}Q_{n+1}^2=O\\
\nonumber
P_n^{1,+}A_{n,n-1}P_{n-1}-(P_{n-2}^{1,+}A_{n-2,n}-P_{n-1}^{1,+}A_{n-1,n})P_n
-P_{n-1}^{1,+}A_{n-1,n+1}P_{n+1}=O.
\end{eqnarray}
At the same time, from (\ref{rel}) it follows that
\begin{eqnarray}
\label{rel12} \nonumber Q_n^1P_n^{+,1}&+&Q_n^{2}P_n^{+,2} \; =P_{n}Q_n^+\\
Q_{n+1}^1P_n^{+,1}&+&Q_{n+1}^{2}P_n^{+,2}=P_{n+1}Q_n^+
\end{eqnarray}
Multiplying the first equation of (\ref{case12}) on the right by
$P_n^{1,+}$, the second equation by $P_n^{2,+}$ and the third,
respectively, by $ -Q_n^+ $ and summing the resulting equations we
obtain, using  (\ref{rel12}) that
\begin{equation*}
P_n^{1,+}A_{n,n-1}(-P_{n-1}Q_n^+ +Q_{n-1}^1 P_n^{1,+} +Q_{n-1}^2
P_n^{2,+})= P_n^{1,+}
\end{equation*}
Thus
\begin{equation*}
A_{n,n-1}(-P_{n-1}Q_n^+ +Q_{n-1}^1 P_n^{1,+} +Q_{n-1}^2
P_n^{2,+})= A_{n,n-1}(-P_{n-1}Q_n^+ + Q_{n-1}P_n^{+})=E,
\end{equation*}
and therefore
\begin{equation*}
  (\g I-A) R_{0:\infty,n}
  = \left[\begin{array}{c} {\bf 0}_{n \times 1}
    \\ E
    \\ {\bf 0}_{\infty \times 1} \end{array}\right].
\end{equation*}
This shows claim~(\ref{mid}). Notice that in the above reasoning
we require that $n \geq r$. A proof for the case $0 \leq n < r$ is
similar, we omit the technical details.
\end{proof}

Now consider the case
\begin{equation}
\label{bound} \sup_{i,j \ge 0} || A_{i,j} || \le C < \infty,
\end{equation}
where $ ||\, . \,|| $ is a certain matrix norm.
 Then the operator $A$ is bounded. Recall that $\g$ is an
element of the resolvent set $\Omega(A)$ if there exists an
operator ${\cal R}(\g)=(\g I-A)^{-1}\in L(l^2_N)$  referred to as
the resolvent of $A$ such that $(\g I-A){\cal R}(\g)u=u$ and
 ${\cal R}(\g)(\g I-A)v=v$ for any $ u $ and $ v\in l^2_N$.
As the operator $A$, the resolvent $ {\cal R}(\g) $ can be
expressed as an infinite matrix with matrix elements of order $N :
{\cal R}(\g)=(\tilde R_{i,j})_{i,j=0}^\infty, \, \tilde R_{i,j}
\in \mathbb{C}^{N \times N} $.

The matrix
\begin{equation*}
{\cal M}(\g, A) \equiv (\tilde
R_{i,j})_{i=0,\dots,r-1}^{j=0,\dots,s-1}
\end{equation*}
is called the Weyl matrix of the operator $A$. Note that the
properties of the Weyl matrix in a more general case of $A$ with
operator elements were studied in \cite{os2}, where it was shown
that $Q(\g)$ and $Q^{+}(\g)$ are the systems of polynomials,
orthogonal with respect to $  {\cal M}(\g, A)$.

 For the bounded
 operators $A$ we may establish the following criterion
for $\Omega(A)$.

\begin{theorem}\label{thm1} Suppose that the band operator $A$ with
  matrix representation (\ref{a}) satisfies (\ref{bound}).
  Then $\g\in \mathbb C$ belongs to the resolvent set of $A$
if and only if there exist positive constants $C,\; q<1$ and a
matrix $\mathfrak M= (\mathfrak
M_{i,j})_{i=1,\dots,r}^{j=1,\dots,s},\, \mathfrak M_{i,j}\in
\mathbb C^{N \times N}$ such that
\begin{equation}
\label{claim} ||R_{k,n}|| \le C q^{|n-k|}, \quad k,n \in \Z_{+},
\end{equation}
where $R_{k,n}$ are defined by (\ref{r1}). In this case, the
matrix $\mathfrak M=\mathfrak M(\g) =
(R_{i,j})_{i=0,\dots,r-1}^{j=0,\dots,s-1} $ is unique, and
coincides with the Weyl matrix ${\cal M}(\g, A)$.
\end{theorem}
\begin{proof}
Necessity. Let $\g \in \Omega(A)$. Assume that $\tilde R=(\tilde
R_{i,j})_{i,j=0}^\infty $ is the matrix representation of ${\cal
R}(\g)$. Take $\mathfrak M = (\tilde
R_{i,j})_{i=0,\dots,r-1}^{j=0,\dots,s-1} $ and consider the matrix
$ R = (R_{k,n})_{k,n=0}^\infty, $ where $ R_{k,n} $ are defined by
(\ref{r1}). From the resolvent identity $(\g I -A) {\cal R}(\g)=I$
together with Lemma~4 and (\ref{init_L}) follows that $ R $ and $
\tilde R $ satisfy the same recurrence relation and the same
initializations; hence they coincide. It remains to show the decay
rate (\ref{claim}). In the scalar case ($A_{i,j} \in \mathbb{C}$)
it follows from the result of \cite{demo} on the decay rate of the
elements of the inverses of band matrices; for $A_{i,j}
\in\mathbb{C}^{N\times N} $ it can be proved in a similar manner.

Sufficiency. Assume that the conditions of the theorem are
satisfied.  As above, we build up the infinite matrix
$R=(R_{k,n})_{k,n=0}^{\infty}$. From Lemma~4 and (\ref{claim}) it
follows that we may correctly define the operator $ (A-\g I)^{-1}
$ on the basis vectors from $l^2_N$ and therefore on the finite
vectors (in this basis). Also, from (\ref{claim}) it follows that
the operator $ (A-\g I)^{-1} $ defined on the finite vectors, is
bounded. Thus we can extend the $ (A-\g I)^{-1} $ on all $l^2_N$
and therefore $ \g \in \Omega(A) $.
\end{proof}

\begin{corollary}
\label{cor1}
If $\g \in \Omega(A)$ and $A$
satisfies the conditions of Theorem 1, then
\begin{equation}
\label{l1} \limsup_{k\to\infty}
||R_k^{j}(\g)||^{\frac{1}{k}}<1,\quad \limsup_{n\to\infty}
||R_n^{i,+}(\g)||^{\frac{1}{n}}<1 ,\quad i=1,\dots,r;\;
j=1,\dots,s.
\end{equation}
\end{corollary}

\begin{theorem}
If $\g \in \Omega(A)$ and  A satisfies the conditions of Theorem
1, then for $i=1,\dots,r$, and $j=1,\dots,s$ there holds
\begin{equation}
\label{l2} \limsup_{k\to\infty}
||Q_k^{i}(\g)||^{\frac{1}{k}}>1,\quad \limsup_{k\to\infty}
||Q_k^{j,+}(\g)||^{\frac{1}{k}}>1 .
\end{equation}
\end{theorem}
\begin{proof}
Multiplying the equation of Lemma 2 on the left and  on the right
with
$$
\left[\begin{array}{cc}
   I_r & -\mathfrak M  \\ {\bf 0}_{s \times r} & I_s
   \end{array}\right]
, \quad \left[\begin{array}{cc}
   I_r & \mathfrak M  \\ {\bf 0}_ {s \times r} & I_s
   \end{array}\right]
$$
gives
  \begin{eqnarray}
  \label{lb}
   && I_{r+s} = \left[\begin{array}{c}
   -R^{+}_{k-r:k+s-1}  \\ Q^{+}_{k-r:k+s-1}
   \end{array}\right]
   \\&&\cdot
   \left[\begin{array}{cc}
      {\bf 0}_{r \times s} &
      - A_{k-r:k-1,k:k+r-1} \\
      A_{k:k+s-1,k-s:k-1} & {\bf 0}_{s \times r}
   \end{array}\right]
   \cdot
   \left[\begin{array}{cc}
   Q_{k-s:k+r-1}  R_{k-s:k+r-1}
   \end{array}\right]\nonumber
\end{eqnarray}
and hence in particular
$$
   Q^+_{k-r:k+s-1}
   \cdot
   \left[\begin{array}{cc}
      {\bf 0}_{r \times s} &
      - A_{k-r:k-1,k:k+r-1} \\
      A_{k:k+s-1,k-s:k-1} & {\bf 0}_{s \times r}
   \end{array}\right]
   \cdot R_{k-s:k+r-1} = I_s .
$$
Therefore
\begin{equation}
\label{idsup} \sum_{m=0}^{s-1}\sum_{n=1}^s
Q_{k+m}^{j,+}A_{k+m,k+m-n}R_{k+m-n}^{j}-
\sum_{m=0}^{r-1}\sum_{n=1}^r
Q_{k+m-n}^{j,+}A_{k+m-n,k+n}R_{k+n}^{j}=E,\; j=1,\dots,s.
\end{equation}
Take $j=1$. Now assume that $\limsup_{k\to\infty}
||Q_k^{1,+}(\g)||^{\frac{1}{k}}\le 1 $. Then for some  $C_1>0$ and
$d, \; q<d<1$ we have $||Q_k^{1,+}(\g)||^{\frac{1}{k}}\le C_1
(1/d)^{k} $. It means that the norm of the left-hand side of
(\ref{idsup}) can be majorated by $ C_2 q^{k-r}(1/d)^{k-r} $ for
some $C_2 > 0 $, which tends to zero as $ k \to \infty $.
Obviously, this contradicts the identity (\ref{idsup}) and
therefore we have proved (\ref{l2}) for $j=1$. By taking
$j=2,\dots, s$ and applying the above arguments we obtain
(\ref{idsup}) for another values of $j$.

Also, from (\ref{lb}) follows
$$
   -R^+_{k-r:k+s-1}
   \cdot
   \left[\begin{array}{cc}
      {\bf 0}_{r \times s} &
      - A_{k-r:k-1,k:k+r-1} \\
      A_{k:k+s-1,k-s:k-1} & {\bf 0}_{s \times r}
   \end{array}\right]
   \cdot Q_{k-s:k+r-1} = I_r .
$$
From this identity we get (\ref{l2}) for $Q_k(\g)$ similarly as we
have done it for $Q_k^{+}(\g)$.

\end{proof}

Finally note that in the scalar case ($ A_{i,j} \in \mathbb{C} $ )
the above results on the operators $A$ were obtained in
(\cite{b2}) for possibly unbounded operators with $(a_k)_{k\geq0}$
defined by $$
 a_k:= \max\{ ||A_{k-r:k-1,k:k+r-1}||,
                       ||A_{k:k+s-1,k-s:k-1} ||\},\quad
                       k\ge 0,
$$ containing a sufficiently dense bounded subsequence. If,
instead we consider the matrix case ($ A_{i,j} \in {\mathbb
C}^{N\times N} $ ), we can obtain the same results as above, we
omit the details.

\end{document}